\documentclass{amsart}
\usepackage{hyperref}
\usepackage[T1]{fontenc}
\usepackage[utf8]{inputenc}
\usepackage{microtype,amsfonts,amssymb,amsthm,tikz-cd}

\newtheorem{proposition}[section]{Proposition}
\newtheorem{lemma}[section]{Lemma}

\theoremstyle{definition}

\newtheorem{instance}[equation]{Example}
\newtheorem{situation}[section]{}

\numberwithin{equation}{section} 
\author[T.-J.~Lee]{Tsung-Ju~Lee}
\address{Tsung-Ju~Lee: Center of Mathematical Sciences and Applications, 20 Garden St., Cambridge, MA 02138, U.S.A.}
\email{tjlee@cmsa.fas.harvard.edu}
\author[B.~Lian]{Bong~H.~Lian}
\address{Bong~H.~Lian, Department of Mathematics, Brandeis University, Waltham MA 02454, U.S.A.}
\email{lian@brandeis.edu}
\author[D.~Zhang]{Dingxin~Zhang}
\address{Dingxin Zhang, Yau Mathematical Sciences Center, Tsinghua University, Beijing 100084, China}
\email{dingxin@tsinghua.edu.cn}

\thanks{This work is supported by the 
Simons Collaboration Grant on Homological Mirror Symmetry and 
Center of Mathematical Sciences and Applications, Harvard University}
\date{\today}
\title{On a conjecture of Huang--Lian--Yau--Yu}

\begin{document}
\begin{abstract}
We verify a formula on the solution rank
of the tautological system arising from ample
complete intersections in a projective homogeneous space
of a semisimple group conjectured by
Huang--Lian--Yau--Yu~\cite{hlyy}.
As an application, we prove the existence of the rank one point 
for such a system,
where mirror symmetry is expected.
\end{abstract}
\maketitle
\section*{Introduction}
\label{sec:org7ba3867}

In this note we verify a formula conjectured by
Huang--Lian--Yau--Yu~\cite{hlyy}. This is a formula on the solution rank of
a tautological system associated with ``ample complete
intersections'' in a projective homogeneous space of a semisimple group.

\begin{situation}
\label{situation:convention}
\textbf{Conventions.}
\begin{enumerate}
\item We work with \emph{complex} algebraic varieties.
\item The cohomology groups will be taken with respect to sheaves of \emph{complex} vector
spaces, or complexes of sheaves of \emph{complex} vector spaces.
\item The notation \(\mathrm{H}^{m}(A,B)\) is potentially confusing: it can mean
the relative cohomology of \(A\) with respect to \(B\), or 
the sheaf cohomology
of \(A\) with coefficients in \(B\). In the situations below, the reader
should distinguish the use by recalling the previously set up notation.
\end{enumerate}
\end{situation}

\begin{situation}
Let \(X\) be an \(n\)-dimensional, smooth, projective variety
with an action of a connected algebraic group \(G\).
Let \(L_1, \ldots, L_r\) be ample invertible sheaves on
\(X\). For a section \(b_i \in \Gamma(X,L_i)\), define \(Y_{b_i}\) to be
the vanishing scheme of \(b_i\). Define \(V=\prod_{i=1}^{r}\Gamma(X,L_i)\).

When \(r = 1\), and \(L_1 = \omega_X^{-1}\), a \(\mathcal{D}_V\)-module \(\tau\), the
\emph{tautological system}, was introduced by Lian--Song--Yau
\cite{lian-song-yau:period-integrals-and-tautological-systems} in order to study
the periods of Calabi--Yau hypersurfaces in \(X\). We shall not need the precise
definition of tautological systems. The upshot is that among its solutions are ``period
integrals'' \(\int_{\Gamma} \Omega\), where \(\Omega\) is some canonically
defined
differential~\cite{lian-yau:period-inegrals-of-cy-and-general-type-complete-intersections}.
We also mention that tautological systems specialize to the Gelfand--Kapranov--Zelevinsky
\(A\)-hypergeometric systems when \(G\) is a torus.

For general \(r\), set \(\mathbb{P}=\mathrm{Proj}(\mathrm{Sym}^{\bullet}(L_1 \oplus \cdots \oplus L_r))\).
Assume further that \(L_{i}\) are \(G\)-equivariant for all \(i\) and
\(L_1 \otimes \cdots \otimes L_r = \omega_X^{\vee}\). Then
\(\mathbb{P}\) admits an action of \(G\times \mathbb{G}_{\text{m}}^{r-1}\), and
the sheaf \(\mathcal{O}_{\mathbb{P}}(r) = \omega_{\mathbb{P}}^{\vee}\). Hence,
one can define a tautological system for \(\mathbb{P}\). This system is what we refer to
as the tautological system of complete intersections.

\emph{When \(X\) is a projective homogeneous variety of a semisimple 
group \(G\)}, Huang,
Lian, Yau, and Yu were able to give a cohomological interpretation
(see~\eqref{eq:bad} below) of the solution space of this tautological system. They
subsequently formulated the following conjecture, which we shall verify in this
note.
\end{situation}

\begin{situation}
\label{situation:conj}
\textbf{Theorem (Conjecture of Huang--Lian--Yau--Yu).}
\emph{Let notation be as above. Assume that \(X\) is a projective homogeneous variety
of a semisimple group \(G\).}
\emph{The solution space of the tautological system \(\tau\) at a point}
\(b=(b_1,\ldots,b_r)\in V\) \emph{is identified with the homology group}
\begin{equation}
\label{eq:good}
\textstyle\mathrm{H}_{n}\big(X-\bigcup_{i=1}^{r}Y_{b_i}\big).
\end{equation}
\end{situation}

\medskip{}
The proof we present is somehow unrelated to the theory of tautological systems, but uses
the cohomological interpretation of Huang--Lian--Yau--Yu to verify their
conjecture. Huang--Lian--Yau--Yu proved that the said solution space can be
identified with
\begin{equation}
\label{eq:bad}
\mathrm{H}_{n+r-1}(\mathbb{U}_b, \mathbb{U}_b \cap \mathbb{D}),
\end{equation}
where \(\mathbb{U}_b\) is the complement of the hypersurface in \(\mathbb{P}\)
determined by \(b\), and \(\mathbb{D}\) the union of fiber-wise coordinate axes.
Our work is to prove \eqref{eq:good} and \eqref{eq:bad} are naturally isomorphic.
The proof of the isomorphy between these groups uses some common methods in
singularity theory and eventually reduces to a simple problem in combinatorics.

In the last section, we give an application of the main theorem. We prove the
existence of a certain special points, called ``rank one points'' in the moduli
space of Calabi--Yau intersections in a Grassmannian. These are the points where
the solution rank of the Picard--Fuchs equations equals one. That is, all, but
one, period integrals of the canonical differential \(\Omega\) (of the
Calabi--Yau complete intersections) become singular. These points are the
candidates of the so-called ``large radius limits'' in the story of mirror
symmetry.

For the ease of exposition we shall be using cohomology instead of homology.
Everything about homology can be deduced by taking duality.

\medskip{}\noindent
\textbf{Acknowledgment.} We are grateful to Professor Shing-Tung Yau, 
the director of CMSA, for his steady encouragement.
We would like to thank Shuai Wang and Chenglong Yu for useful
communications. We would like to 
thank the anonymous referee for his/her careful reading and useful comments.
We also thank CMSA for providing a pleasant working environment. 
The presented work is supported by the Simons Collaboration Grant on Homological Mirror Symmetry and Applications 2015-2022.
\section*{Some general facts about cohomology}
\label{sec:org164cc14}
\begin{situation}
\label{situation:resolution}
This paragraph discusses how to compute relative cohomology.
The notation we use shall be in agreement of the notation employed later in the
note.
Let \(E\) be a smooth complex manifold.
Let \(D = \bigcup_{i\in I}D_i\) be a simple normal crossing divisor on \(E\).
Let \(i: D \to E\) be the inclusion map.
For \(J \subset I\), define \(D_J = \bigcap_{j\in J} D_{j}\).
Define \(D^{(m)} = \coprod_{\# J=m} D_J\). Let \(i_m: D^{(m)} \to E\) be the natural map.
Then there is a resolution
\begin{equation}
\label{eq:resolution-snc}
i_{\ast}\mathbb{C}_{D} \to [ i_{1\ast}\mathbb{C}_{D^{(1)}} \to i_{2\ast}\mathbb{C}_{D^{(2)}} \to \cdots ]
\end{equation}
Let \(j: U \to E\) be the open embedding of the complement of \(D\). Then there
is an exact sequence
\begin{equation*}
0 \to j_{!}\mathbb{C}_U \to \mathbb{C}_E \to i_{\ast}\mathbb{C}_D \to 0.
\end{equation*}
It follows from~\eqref{eq:resolution-snc} that we have an exact sequence
\begin{equation*}
j_{!}\mathbb{C}_{U} \to \mathbb{C}_E \to i_{1\ast}\mathbb{C}_{D^{(1)}} \to i_{2\ast}\mathbb{C}_{D^{(2)}} \to \cdots.
\end{equation*}

Now let \(u: F \to E\) be a closed embedding of a complex submanifold of \(E\).
Let \(v: U\cap F \to F\) be the inclusion,
and let \(F_i = D^{(i)} \times_{E} F = \coprod_{\#J=i} F \cap D_J\).
As an abuse of notation, we shall denote the induced map \(F_m \to F\) also by
\(i_m\).
By the exactness of \(u^{-1}\) we get an exact sequence
\begin{equation*}
u^{-1}j_{!}\mathbb{C}_U \to \mathbb{C}_{F} \to u^{-1}i_{1\ast} \mathbb{C}_{D^{(1)}} \to \cdots.
\end{equation*}
By proper base change, we know
\begin{equation*}
u^{-1}j_{!}\mathbb{C}_U = v_{!}\mathbb{C}_{U\cap F}.
\end{equation*}
Since \(i_m\) is proper, \(i_{m\ast} = i_{m!}\). By proper base change we also
have \(u^{-1}i_{m\ast} = i_{m\ast}\). we conclude that we have an exact sequence
\begin{equation*}
v_{!}\mathbb{C}_{U\cap F} \to \mathbb{C}_{F} \to i_{1\ast}\mathbb{C}_{F_1} \to \cdots.
\end{equation*}
By definition, \(\mathrm{H}^{m}(F,F\cap D) = \mathrm{H}^{m}(F,v_{!}\mathbb{C}_{U\cap F})\).
Thus we deduce that
\begin{equation*}
\mathrm{H}^{m}(F,F \cap D) =
\mathbb{H}^{m}(F,\mathbb{C}_{F}\to i_{1\ast}\mathbb{C}_{F_1}\to\cdots).
\end{equation*}
\end{situation}

\begin{situation}
\label{situation:bundle-notation}
Let \(B\) be a smooth complex algebraic variety of pure dimension \(n\). Let
\(L_1, \ldots, L_r\) be invertible sheaves on \(B\). Let
\begin{equation*}
\mathbb{P} = \mathrm{Proj}(\mathrm{Sym}^{\bullet}(L_1 \oplus \cdots \oplus L_r))
\end{equation*}
be the projective space bundle associated with the direct sum of the \(L_i\)'s.
Let \(E\) be the geometric vector bundle associated with
\(L_1^{\vee}\oplus\cdots\oplus L^{\vee}_r\). Then \(E - \zeta(B)\), where
\(\zeta:B \to E\) is the zero section, is a \(\mathbb{G}_{\text{m}}\)-torsor on
\(\mathbb{P}\).

Since \(E\) is a direct sum of line bundles, it makes sense to talk about
whether the \(i\)th factor of a point on \(E\) is zero. Let \(D_i\) be the
divisor of \(E\) consisting of these points. The
projection \(D_i \to B\) is the geometric vector bundle associated with the
direct sum \(\bigoplus_{k\neq i}L_k^{\vee}\).

More generally, for a subset \(J \subset \{1,2,\ldots,r\}\), we define
\(D_J = \bigcap_{j\in J} D_j\). The projection \(D_J \to B\) is thus the
geometric vector bundle associated with the direct sum
\(L_J = \bigoplus_{k\notin J} L_k\).
\end{situation}

\begin{situation}
Let the notation be as in~\S\ref{situation:bundle-notation}.
Let \(b=(b_1,\ldots,b_r)\in\mathrm{H}^0(B,\bigoplus_{i=1}^{r}L_i)\) be a
\emph{nonzero} section.
\begin{enumerate}
\item The section \(b\) gives rise to a section of the invertible sheaf
\(\mathcal{O}_{\mathbb{P}}(1)\) on \(\mathbb{P}\). The \emph{nonvanishing locus}
of this section is denoted by \(\mathbb{U}_b \subset \mathbb{P}\). Then
\(\mathbb{U}_b\) is an algebraic variety of pure dimension \(n + r -1\).
\item Let \(\mathbb{D}_i \subset \mathbb{P}\) be the image of \(D_i\).
\item The section \(b\) gives rise to a function \(\varphi_b: E \to
   \mathbb{A}^1\) by fiber-by-fiber pairing:
\[\varphi(b)(\ell_1,\ldots,\ell_r)=\sum_{i=1}^{r}\langle\ell_i,b_i(x)\rangle.\]
Let \(F_b = \varphi_b^{-1}(1)\). We mention in passing that \(F_b\) is an
analogue of the Milnor fiber of a quasi-homogeneous singularity.
\item Let \(E_{i}\) be the closed subvariety of \(F_b\) defined by the vanishing of
\(\ell_i\). Thus the \(E_i\)'s are the intersection of \(F_b\) with a
collection of divisors on \(E\) with normal crossings.
\item Let \(F_1 = \coprod E_i\), \(F_2 = \coprod E_i \cap E_j\), and so on.
\item Let \(Y_{i}\) be the hypersurface in \(B\) defined by the vanishing locus of
\(b_i\). Let \(U_i\) be the complement of \(Y_i\).
Let \(U_b = \bigcup_{i=1}^{r} U_i\)
\end{enumerate}
\end{situation}

\begin{lemma}
The projection \(F_{b} \to \mathbb{P}\) induces an isomorphism
\(F_b \cong \mathbb{U}_b\) of algebraic varieties. Hence the pair
\((F_b, \bigcup E_i)\) is isomorphic to the pair
\((\mathbb{U}_b,\mathbb{U}_b\cap \mathbb{D})\).
\end{lemma}

\begin{proof}
The problem being local, we can assume that the \(L_i\)'s are trivial bundle on
\(B\). Then the \(b_i\)'s are given by a collection of regular functions on
\(B\). In this case, \(F_b\) is given by
\begin{equation*}
\left\{(x,\ell_1,\ldots,\ell_r) \in B \times \mathbb{A}^r: \sum b_i(x)\ell_i = 1\right\}
\end{equation*}
whereas
\begin{equation*}
\mathbb{U}_b = \left\{(x, [u_1,\ldots,u_r]) \in B \times \mathbb{P}^{r-1} : \sum b_i(x) u_i \neq 0\right\}.
\end{equation*}
It is easy to see then that the natural map \(F_b \to \mathbb{U}_b\) sending
\((x,\ell_1,\ldots,\ell_r)\) to \((x,[\ell_1,\ldots,\ell_r])\) is well-defined
and is an isomorphism of algebraic varieties.
\end{proof}

The lemma above shows that the validity of Theorem~\ref{situation:conj}
will follow from the following proposition.

\begin{proposition}
\label{proposition:isomorphy}
There is an isomorphism
\begin{equation*}
\mathrm{H}^{m}(F_b, {\textstyle\bigcup E_i}) \cong 
\mathrm{H}^{m-r+1}(U_1\cap\cdots\cap U_r).
\end{equation*}
\end{proposition}

The first step of the proof the proposition needs the explicit form of Caylay's
trick.

\begin{lemma}
\label{lemma:affine-bundle}
The projection \(F_b \to U_b\) is an \(\mathbb{A}^{r-1}\)-bundle.
In particular, \(F_b\) and \(U_b\) have the same homotopy type.
\end{lemma}

\begin{proof}
The problem being local, we can assume the \(L_i\)'s are trivial. Then
\(F_b\) is defined by \(\sum b_i(x)\ell_i = 1\). If \(U_{k}\) is the open
subset of \(X\) consisting of \(x \in B\) such that \(b_k(x) \neq 0\). Then
\(U_{k} \subset U_b\) and we can define an open embedding
\begin{equation*}
\varphi_k : U_k \times \mathbb{A}^{r-1} \to F_b
\end{equation*}
by
\begin{equation*}
(x,\ell_1,\ldots,\widehat{\ell_k},\ldots,\ell_r) \mapsto
\left(x, \ell_1, \cdots, \ell_{k-1}, \frac{1 - \sum_{i\neq k} b_i(x)\ell_i}{b_{k}(x)},\ell_{k+1},\ldots,\ell_r \right).
\end{equation*}
Clearly the union of the images of \(\varphi_k\) equals \(F_b\). This completes
the proof.
\end{proof}

Now we turn to the proof of Proposition~\ref{proposition:isomorphy}. \refstepcounter{section}
\section*{The proof}
\label{sec:org89a7431}

\noindent
\textbf{\thesection.}
In the sequel, in order to avoid using the notion of ``complexes in the derived
category'', we shall insist on working with the abelian category of \emph{complex of
sheaves}. Cohomology groups are taken only in the last step. We shall also now
employ the notation of \S\ref{situation:resolution}, with \(F = F_b\) (the other
notation are compatible with the situation of \S\ref{situation:resolution}). The
groups \(\mathrm{H}^{\bullet}(F_b,\bigcup E_i)\) are the hypercohomology groups
of the complex
\begin{equation*}
\mathbb{C}_{F} \to i_{1\ast}\mathbb{C}_{F_1} \to i_{2\ast} \mathbb{C}_{F_2} \to \cdots.
\end{equation*}
Let \(\pi: F_b \to U_b\) be the natural projection, which is a smooth,
affine morphism that is a torsor of a vector bundle. Then the above
hypercohomology can be computed on \(U_b\) by taking direct images:
\begin{equation}
\label{eq:direct-image-q}
R\pi_{\ast}\mathbb{C}_F \to R\pi_{\ast} i_{1\ast} \mathbb{C}_{F_1} \to R\pi_{\ast}i_{2\ast}\mathbb{C}_{F_2} \to \cdots.
\end{equation}
Again, we recapitulate that the items
\begin{equation}
\label{eq:complex-vs-derivedcat}
R\pi_{\ast} i_{k\ast} \mathbb{C}_{F_k}
\end{equation}
are to be thought as complexes of sheaves, and the above displayed equation
should be thought as a double complex, or a complex in the abelian category of
complexes of sheaves. Precisely, we fix an injective resolution \(I^{\bullet}\)
of \(\mathbb{C}_{F_i}\) and the item~\eqref{eq:complex-vs-derivedcat} is
regarded as the complex \(\pi_{\ast}I^{\bullet}\).

Let \(J\) be a subset of \(\{1,2,\ldots,r\}\).
Let \(\overline{J}\) be the complement of \(J\) in \(\{1,2,\ldots,r\}\).
Let \(U_{J}=\bigcup_{k\in J} U_k\). Then we know from
Lemma~\ref{lemma:affine-bundle} that
\begin{equation*}
E_{i_1}\cap\cdots \cap E_{i_s} \to U_{\overline{\{i_1,\ldots,i_s\}}}
\end{equation*}
is an affine space bundle, hence the direct image of the fixed injective
resolution of \(\mathbb{C}_{E_{i_1}\cap\cdots \cap E_{i_s}}\) becomes an
injective resolution of \(\mathbb{C}_{U_{\overline{\{i_1,\ldots,i_s\}}}}\).
It follows that the complex~\eqref{eq:direct-image-q} is of the form
\begin{equation}
\label{eq:pseudo-mv}
\mathbb{C}_{U_b} \to \bigoplus_{\#J=1} Rj_{\overline{J}\ast}\mathbb{C}_{U_{\overline{J}}}
\to \bigoplus_{\#J=2} Rj_{\overline{J}\ast}\mathbb{C}_{U_{\overline{J}}} \to \cdots \to Rj_{1\ast}\mathbb{C}_{U_1}\oplus \cdots \oplus Rj_{r\ast}\mathbb{C}_{U_r}
\end{equation}
where \(j_{J}\) is the inclusion of \(U_J\) into \(U_b\).

We have reduced our problem to a topological one. Thus we will work in the category of
locally compact, Hausdorff, topological spaces.
The situation is that we are given an open covering
\(U_b = U_1 \cup \cdots \cup U_r\) of topological spaces, and
let \(\iota: U_1 \cap \cdots \cap U_r \to U_b\) be the open immersion of the
deepest intersection. Proving Proposition~\ref{proposition:isomorphy} reduces
to proving the (double) complex~\eqref{eq:pseudo-mv} and the complex
\(R\iota_{\ast}\mathbb{C}_{U_1 \cap \cdots \cap U_r}[1-r]\) have the same
hypercohomology.

\medskip{}\noindent
\textbf{Notation.}
For the later combinatorial manipulation, from now on we use
\((i_1,\ldots,i_m)\) to denote the complex
\begin{equation*}
Rj_{i_1,\ldots,i_m\ast}\mathbb{C}_{U_{i_1,\ldots,i_m}}.
\end{equation*}
More generally, let
\begin{equation*}
(i^{(1)}_1 \cap \cdots \cap i^{(1)}_{s_1}, i^{(2)}_1 \cap \cdots \cap i^{(2)}_{s_2} ,
\ldots, i^{(m)}_m \cap \cdots \cap i^{(m)}_{s_m})
\end{equation*}
be the direct image, to \(U_b\), of the constant sheaf on the open subset
\begin{equation*}
(U_{i^{(1)}_1}\cap \cdots\cap U_{i^{(1)}_{s_1}}) \cup \cdots \cup (U_{i^{(m)}_{1}} \cap\cdots \cap U_{i^{(m)}_{s_m}}).
\end{equation*}
Moreover, we shall use the usual addition to denote direct sum. Therefore, our
final task is to prove the following exercise in topology.

\begin{proposition}
\label{proposition:final}
The (double) complex
\begin{equation*}
\textstyle
(1,\ldots,r) \to \sum_i (\overline{i}) \to \sum_{i < j} (\overline{i,j}) \to \cdots \to \sum_i (i)
\end{equation*}
admits a chain map to \((1\cap \cdots \cap r)[-(r-1)]\) that is a quasi-isomorphism.
\end{proposition}

The proof is by induction. We first deal with the cases when \(r = 2\) and
\(r=3\), which are also bases of the inductive proof.

\begin{instance}
\label{example:requal2}
Assume \(r = 2\).
Then \(U_b = U_1 \cup U_2\).
The complex~\eqref{eq:direct-image-q} is
\begin{equation*}
(1,2) \to (1) + (2)
\end{equation*}
But this is clearly quasi-isomorphic to the complex \((1\cap 2)[-1]\), by the
Mayer--Vietories principle.
\end{instance}

\begin{instance}
\label{example:requal3}
Assume that \(r = 3\). Then the complex~\eqref{eq:direct-image-q} is
\begin{equation}
\label{eq:case3}
(1,2,3) \to (1,2) + (1,3) + (2,3) \to (1) + (2) + (3).
\end{equation}
Consider
\begin{equation*}
\begin{tikzcd}
&  (2,3)\ar[d]\ar[r,equal]  &(2,3)  \ar[d] &  \\
(1,2,3) \ar[r] & (1,2) + (1,3) + (2,3) \ar[r] & (1) + (2) + (3)
\end{tikzcd}
\end{equation*}
where the left vertical arrow is given by the inclusion and 
the right vertical one is given by the Mayer--Vietories of
\(U_2 \cup U_3\) with respect to the covering \(U_2\) and
\(U_3\). Taking the quotient complex yields a complex quasi-isomorphic
to~\eqref{eq:case3}, which is
\begin{equation*}
(1,2,3) \to (1,2) + (1,3) \to (1) + (2\cap 3).
\end{equation*}
As a second step, we consider the item \((1,2,3)\) on the left. 
Form the diagram
\begin{equation*}
\begin{tikzcd}
(1,2,3) \ar[r,equal] \ar[d,equal] & (1,2,3) \ar[d] & \\
(1,2,3) \ar[r] & (1,2)+(1,3) \ar[r] & (1) + (2\cap 3).
\end{tikzcd}
\end{equation*}
The right vertical arrow being induced by the Mayer--Vietories of
\(U_1 \cup U_2\cup U_{3}\) with respect to the covering \(U_{1}\cup U_{2}\)
abd \(U_{1}\cup U_{3}\). Thus taking the
quotient yields a quasi-isomorphic complex which is
\begin{equation*}
0 \to (1,2\cap 3) \to (1) + (2\cap 3).
\end{equation*}
By the Mayer--Vietories for the space \(U_1 \cup (U_2 \cap U_3)\), this complex
equals \((1\cap 2\cap 3)[-2]\). We win.
\end{instance}

\begin{proof}[Proof of Proposition~\ref{proposition:final}]
The pattern can already be seen in
the above example when \(r =3\): we modify, from right to left, extra terms.
Starting with the complex
\begin{equation*}
(1,2,\ldots,r) \to \sum_{i=1}^{r} ( 1,\ldots,\widehat{i},\ldots,r) \to \cdots
\to \sum_{1\le i<j\le r} (i,j) \to \sum_{i=1}^{r} (i),
\end{equation*}
we first modify the right-most item by considering
\begin{equation*}
\begin{tikzcd}
& & (r-1,r)\ar[d] \ar[r,equal] & (r-1,r) \ar[d] & \\
&\cdots \ar[r] &\sum_{1\le i<j\le r} (i,j) \ar[r] &\sum_{i=1}^{r} (i).
\end{tikzcd}
\end{equation*}
Using the Mayer--Vietories for \(U_{r-1}\cap U_{r}\)
we see the quotient complex, which is quasi-isomorphic to the started one, is of
the form
\begin{equation*}
C_1: (1,\ldots,r) \to \cdots \to \sum_{1\le i<j\le r-2} (i,j)+ 
\sum_{i=1}^{r-2} [(i,r-1)+(i,r)] \to \sum_{i}^{r-2} (i) + (r-1\cap r).
\end{equation*}
Now look at the third from the last term in \(C_{1}\);
namely \(\sum_{1\le i<j<k\le r} (i,j,k)\), and
let \(Q_{1} = \sum_{1\le i\le r-2} (i,r-1,r)\) which 
contains those terms where both indices \(r-1\) and \(r\)
appear. We form the diagram
\begin{equation*}
\begin{tikzcd}[column sep=1em]
& & Q_{1}\ar[d] \ar[r,equal] & Q_{1} \ar[d] & & \\
&\cdots \ar[r] &\displaystyle\sum_{1\le i<j<k\le r} (i,j,k) \ar[r] &
\displaystyle\sum_{1\le i<j\le r-2} (i,j)+ 
\sum_{i=1}^{r-2}[(i,r-1)+(i,r)]\ar[r] &\cdots. 
\end{tikzcd}
\end{equation*}
Each summand of \(Q_1\), say \((i,r,r-1)\), fits into an exact sequence
\begin{equation*}
(i,r,r-1) \to (i,r) + (i,r-1) \to (i, r\cap r-1).
\end{equation*}
Using these sequences we can replace the complex \(C_1\) by its quasi-isomorphic
quotient \(C_2\), and eliminate all the terms appearing in \(Q_1\). The resulting complex
is 
\begin{align*}
C_2: (1,\ldots,r) &\to \cdots \to \sum_{1\le i<j<k\le r-2}(i,j,k)+
\sum_{1\le i<j\le r-2}[(i,j,r-1)+ (i,j,r)]\\ 
&\to \sum_{1\le i<j\le r-2}(i,j) + \sum_{1\le i\le r-2} (i,r-1\cap r)\to 
\sum_{i}^{r-2} (i) + (r-1\cap r).
\end{align*}
Next we eliminate the term
\begin{equation*}
Q_{2} = \sum_{1\le i<j\le r-2} (i,j,r-1,r)
\end{equation*}
in the forth from the last term in \(C_{2}\). We then achieve the complex
\begin{align*}
C_3: (1,\ldots,r) &\to \cdots \to \sum_{1\le i<j<k<l\le r-2}(i,j,k,l)+
\sum_{1\le i<j<k\le r-2} [(i,j,k,r-1)+(i,j,k,r)]\\
&\to \sum_{1\le i<j<k\le r-2}(i,j,k)+
\sum_{1\le i<j\le r-2}(i,j,r-1\cap r)\\ 
&\to \sum_{1\le i<j\le r-2}(i,j) + \sum_{1\le i\le r-2} (i,r-1\cap r)\to 
\sum_{i}^{r-2} (i) + (r-1\cap r).
\end{align*}
Inductively, we can finally achieve the complex
\begin{align*}
C_{r-1}: 0 &\to (1,\ldots,r-2,r-1\cap r)\\
&\to\cdots\\
&\to \sum_{1\le i<j<k<l\le r-2}(i,j,k,l)+
\sum_{1\le i<j<k\le r-2} (i,j,k,r-1\cap r)\\
&\to \sum_{1\le i<j<k\le r-2}(i,j,k)+
\sum_{1\le i<j\le r-2}(i,j,r-1\cap r)\\ 
&\to \sum_{1\le i<j\le r-2}(i,j) + \sum_{1\le i\le r-2} (i,r-1\cap r)\\
&\to 
\sum_{1\le i\le r-2} (i) + (r-1\cap r).
\end{align*}
By induction on \(r\), the complex \(C_{r-1}\) is the complex
associated with the the covering \(U_{1},\ldots,U_{r-2},U_{r-1}\cap U_{r}\) shifted by \(1\),
which by induction hypothesis computes 
\begin{equation*}
U_{1}\cap\cdots \cap U_{r}[-(r-2)-1]=U_{1}\cap\cdots\cap U_{r}[-(r-1)]
\end{equation*}
as desired.
\end{proof}
\section*{An application}
\label{sec:org0145e49}
\begin{situation}
\label{situation:rank-1-point}
Let \(X = G/P\) be a projective homogeneous space of a semisimple algebraic
group. Let \(L_1, \ldots, L_r\) be a collection of \(G\)-equivariant invertible sheaves on \(X\)
such that \(L_1 \otimes \cdots \otimes L_r = \omega_X^{\vee}\). Then a complete
intersection in \(X\) with respect to \(L_1, \ldots, L_r\) is a Calabi--Yau
variety.

Let \(V = \prod \mathrm{H}^{0}(X,L_i)\). As we have mentioned, there is a
\(\mathcal{D}_V\) module \(\tau\), the tautological system, whose local solution space is
identified with a homology group, as in the main theorem. We say a point \(b \in V\) is a rank one
point, if the space of formal power series solution of \(\tau\) at \(b\) is
1-dimensional. These points are important in the classical story of ``mirror
symmetry''. As observed by
Huang--Lian--Zhu~\cite{huang-lian-zhu:period-integral-and-riemann-hilbert-correspondence},
and independently by S.~Bloch, the cohomological description of a
tautological system can be used in the search of rank one points.
\end{situation}

\begin{proposition}
\label{proposition:rank-1-calabi-yau-complete-intersection-grass}
Let the notation be as in~\S\ref{situation:rank-1-point}.
Assume further \(X = \mathrm{Grass}(d,N)\) is a Grassmannian acted by
\(G=\mathrm{SL}_{N}\). Then \(\tau\) admits a rank one point.
\end{proposition}

\begin{proof}
Write \(L_i = \mathcal{O}_X(d_i)\), where
\(\mathcal{O}_X(1)\) is the generator of the Picard group of \(X\)
(which defines the Plücker embedding of \(X\)).
It is well-known that \(\omega_X = \mathcal{O}_X(-N)\).
Consider, in terms of the Plücker
coordinate, the hypersurface \(\Pi\) defined by
\begin{equation*}
x_{1,2,\ldots,d} \cdot x_{2,3,\ldots,d+1} \cdots x_{N,1,\ldots, d-1} = 0.
\end{equation*}
Then we have \(\mathcal{O}_{X}(-\Pi) = \omega_X\).
Since \(\bigotimes L_i=\omega_X^{\vee}\), we have \(\sum d_i=N\). Define
\begin{align*}
b_1^{(0)} = x_{1,2,\ldots,d} \cdots x_{d_{1},d_{1}+1,...,d_1+d} &\in \mathrm{H}^0(X,\mathcal{O}_X(d_1)), \\
&\vdots \\
b_r^{(0)} = x_{N-d_{r}+1, \ldots, N} \cdots x_{N,1,\ldots,d-1} &\in  \mathrm{H}^0(X,\mathcal{O}_X(d_r)).
\end{align*}
We get a point \(b^{(0)} = (b^{(0)}_1,\ldots,b^{(0)}_r) \in V\). Let
\(Y_{i}^{(0)}\) be the vanishing locus of \(b_i^{(0)}\). Then \(\Pi\) equals the
union of all \(Y_{i}^{(0)}\).
By~\cite[Proposition~8.6]{huang-lian-zhu:period-integral-and-riemann-hilbert-correspondence},
\(\dim \mathrm{H}_{\dim X}(X-\Pi) = 1\). This, together with the main theorem,
implies that, up to scaling, \(\tau\) admits a unique nonzero formal solution around
\(b^{(0)} \in V\). Thus \(b^{(0)}\) is a rank one point in the parameter space
\(V\) of Calabi--Yau complete intersections in \(X = \mathrm{Grass}(d,N)\).
\end{proof}


\begin{thebibliography}{1}

\bibitem{hlyy}
  An~Huang, Bong H.~Lian, Shing-Tung Yau, and Chenglong Yu.
  \newblock Period integrals of local complete intersections and tautological
  systems. \texttt{arXiv:1801.01194}.
  \newblock 2018.

\bibitem{huang-lian-zhu:period-integral-and-riemann-hilbert-correspondence}
  An~Huang, Bong~H.~Lian, and Xinwen Zhu.
  \newblock Period integrals and the {R}iemann--{H}ilbert correspondence.
  \newblock {\em J.~Differential Geom.}, 104(2):325--369, 2016.

\bibitem{lian-song-yau:period-integrals-and-tautological-systems}
  Bong~H.~Lian, Ruifang Song, and Shing-Tung Yau.
  \newblock Periodic integrals and tautological systems.
  \newblock {\em J.~Eur.~Math.~Soc.~(JEMS)}, 15(4):1457--1483, 2013.

\bibitem{lian-yau:period-inegrals-of-cy-and-general-type-complete-intersections}
  Bong~H.~Lian and Shing-Tung Yau.
  \newblock Period integrals of {CY} and general type complete intersections.
  \newblock {\em Invent.~Math.}, 191(1):35--89, 2013.

\end{thebibliography}
\end{document}